\documentclass[preprint]{elsarticle}

\usepackage{graphicx}

\usepackage{amssymb}
\usepackage{amsmath}
\usepackage{amsthm}
\usepackage{color}

 \newtheorem{theo}{Theorem}[section]

\newtheorem{prop}[theo]{Proposition}
\newtheorem{rema}[theo]{Remark}

\def\bN{{\mathbb N}}

\def\bR{{\mathbb R}}

\def\bS{{\mathbb S}}

\def\NN{{\mathbb N}}

\def\RR{{\mathbb R}}

\def\SS{{\mathbb S}}

\def\e{\eqno}
\def\la{\langle}
\def\ra{\rangle}

\let\a=\alpha
\let\b=\beta

\let\e=\epsilon
\let\lam=\lambda
\let\vp=\varphi

\let\s=\sigma

\let\D=\Delta

\newcommand\re{\mathop{\rm Re}\, }

\let\dis=\displaystyle

\newcommand\cF{\mathcal F}

\newcommand\cS{\mathcal S}
\newcommand\cM{\mathcal M}

\newcommand\cK{\mathcal K}

\let\wt=\widetilde

\def\ve {\mathbf e}

\def\bfe{\mathbf e}

\date{today}
\begin{document}
\title%[Characterization   of  probability measure]
{A characterization   of  probability measure with finite moment and an application to
the Boltzmann equation  %without angular cutoff
}
\author[C]{Yong-Kum Cho}
\address[C]{Department of Mathematics, Chung-Ang University, Seoul 156-756, Korea}
\ead[C]{ykcho@cau.ac.kr}
\author[M]{Yoshinori Morimoto }
\address[M]{Graduate School of Human and Environmental Studies,
Kyoto University,
%\newline\indent
Kyoto, 606-8501, Japan} \ead[M]{morimoto@math.h.kyoto-u.ac.jp}
\author[W]{Shuaikun Wang }
\address[W]{Department of Mathematics, City University of Hong Kong,
Hong Kong, \\P. R. China}
\ead[W]{shuaiwang4-c@my.cityu.edu.hk}
\author[Y]{Tong  Yang \corref{cor1}}
\address[Y]{Department of Mathematics, City University of Hong Kong,
Hong Kong, \\ P. R. China
} 
\ead[Y]{matyang@cityu.edu.hk}

%\cortext[cor1]{corresponding author}

%\fntext[fn1]{Supported in part 
%by  Grant-in-Aid for Scientific Research No.25400160,
%Japan Society of the Promotion of Science. } 
%
%\fntext[fn3]{Supported in part by the General Research Fund of Hong Kong, CityU No.104511, and the Shanghai Jiao Tong University.}
%


\date{}

\begin{abstract}
We characterize probability measure with finite moment of any order in terms of the symmetric difference 
operators of their Fourier transforms. By using our new characterization, we prove the 
continuity $f(t,v) \in C((0, \infty); L^1_{2k-2 +\a})$, where $f(t, v)$ stands for the density of unique measure-valued solution $(F_t)_{t\ge0}$ of the Cauchy problem for the 
homogeneous non-cutoff Boltzmann equation, with Maxwellian molecules, corresponding to a probability measure
 initial datum $F_0$ satisfying 
\[
\int |v|^{2k-2+\a} dF_0(v) < \infty, \enskip 0\le \a < 2, \enskip k= 2, 3, 4,\cdots,
\]
provided that $F_0$ is not a single Dirac mass.

%In the present note, under the non-cutoff assumption we improve the previous result 
%$f(t,v) \in L_{loc}^\infty((0, \infty); L^1_{2k-2 +\a})$  in \cite{MWY-2015} to 
%by using this characterization. Here
%$f(t,v)$ is a unique solution of the Cauchy problem for the spatially homogeneous Boltzmann equation
%of Maxwellian molecules with an initial datum $F_0 \in P_0(\RR^3)$ ( $ F_0 \ne \delta$)
%satisfying
%\[
%\int |v|^{2k-2+\a} dF_0(v) < \infty, \enskip 0< \a < 2, \enskip k= 2, 3, 4,\cdots.
%\]

%\vskip0.3cm
%
%\noindent
%{\bf R\'esum\'e} 

\noindent

\end{abstract}
\maketitle
%\tableofcontents

\section{Introduction}\label{s1}

The spatially homogeneous Boltzmann equation states
\begin{equation}\label{bol}
\partial_t f(t,v) %+v\cdot\nabla_x f
=Q(f, f)(t,v),%\quad f(0,v)=f_0(v),
\end{equation}
where $f(t,v)$ is the density distribution of particles with 
velocity $v \in \RR^3$ at time $t$. The most interesting and
important part of this equation is the collision operator
given on the
 right hand side that captures the change rates of the density distribution
through elastic binary collision:
\[
Q(g, f)(v)=\int_{\RR^3}\int_{\mathbb S^{2}}B\big({v-v_*},\sigma
\big)
 \big\{g(v'_*) f(v')-g(v_*)f(v)\big\}d\sigma dv_*\,,
\]
where 
$$
v'=\frac{v+v_*}{2}+\frac{|v-v_*|}{2}\sigma,\,\,\, v'_*
=\frac{v+v_*}{2}-\frac{|v-v_*|}{2}\sigma
$$
for each $\sigma \in \SS^2$, which follows {from} the conservation of momentum and energy,
\[ v' + v_*' = v+ v_*, \enskip |v'|^2 + |v_*'|^2 = |v|^2 + |v_*|^2.
\]
% between pre-collisional velocities 
% $(v',v_*')$ and
% post- collisional velocities  $(v,v_*)$.

%As in \cite{morimoto-2012, YM}, we consider 
%The non-negative cross section
% $B(z, \sigma)$ depends only on $|z|$ and the scalar product
%$\frac{z}{|z|}\,\cdot\, \sigma$.  
Motivated by some physical models, we assume that the non-negative cross section 
$B$  takes the form
\begin{equation*}
B(|v-v_*|, \cos \theta)=\Phi (|v-v_*|) b(\cos \theta),\,\,\,\,\,
\cos \theta=\frac{v-v_*}{|v-v_*|} \, \cdot\,\sigma\, , \,\,\,
0\leq\theta\leq\frac{\pi}{2},
\end{equation*}
where
\begin{align}%\label{1.2-0}
&\Phi(|z|)=\Phi_\gamma(|z|)= |z|^{\gamma} \enskip \mbox{for some $\gamma>-3$ and }   \notag \\
& b(\cos \theta)\theta^{2+\nu}\ \rightarrow K\ \
 \mbox{as} \ \ \theta\rightarrow 0+  \enskip 
\mbox{for some $0<\nu<2$ and $K>0$. }\label{singular-cross}
\end{align}
Throughout this paper, we will only consider the case when
\[
\gamma = 0,\enskip 0<\nu<2\,,
\]
which is called the Maxwellian molecule type cross section, because the analysis
relies on a simpler form of the equation after taking Fourier transform
in $v$ by the Bobylev formula. 
As usual, we may restrict  $\theta \in [0,\pi/2]$ by considering 
 {\it symmetrized}  cross section 
\[
[ b(\cos \theta)+b(\cos (\pi-\theta))]{\bf 1}_{0\le \theta\le \pi/2}
\]
(see \cite{villani2}). In what follows, we make a slightly general assumption on the cross section that
\begin{align}\label{index-sing}
 \sin^{\a_0} (\theta/2) b(\cos \theta) 
\sin \theta  \in L^1((0, \pi/2]) \enskip \mbox{for some} \enskip \alpha_0 \in (0, 2),
\end{align}
which is fulfilled for $b$ with \eqref{singular-cross} if $\nu < \alpha_0$.

For $k\in\bN^+,\a\in[0,2),k+\a>1$, we denote by $P_{2k-2+\a}(\RR^d)$ the  set of all probability measure  $F$ on $\RR^d$, $d \ge 1$, such that
\[
\int_{\RR^d} |v|^{2k-2+\a} dF(v) < \infty\,
\] 
and %moreover  it requires that 
\begin{align}\label{mean}
\int_{\RR^d} v_j dF(v) = 0, \enskip j =1,\cdots d\,,
\end{align}
when $2k-2+\a \ge 1$.
In this paper, we consider
the Cauchy problem of \eqref{bol} with initial datum
\begin{equation}\label{initial}
f(0,v) = F_0 \in P_{2k-2+\a}.
\end{equation}

Let $\cF$ denote the Fourier transform operator defined as 
\[
\cF(F)(\xi) =\hat F(\xi) = \int_{\RR^d} e^{-iv\cdot \xi} dF(v)
\]
for each $F \in P_0(\RR^d)$, called the characteristic function of $F$, and  $\cK = \cF(P_0(\RR^d))$.
Inspired by a series of works by Toscani and his co-authors \cite{Carlen-Gabetta-Toscani, Gabetta-Toscani-Wennberg, toscani-villani}, Cannone-Karch introduced the space 
$\cK^\a$ for $\a > 0$  defined as follows:
\begin{align}\label{K-al}
\cK^\alpha =\{ \varphi \in \cK\,;\, \|\varphi - 1\|_{\alpha} < \infty\}\,,
\end{align}
where 
\begin{align}\label{dis-norm}
\|\varphi - 1\|_{\alpha} = \sup_{\xi \in \RR^d} \frac{|\varphi(\xi) -1|}{|\xi|^\alpha}.
\end{align}
The space $\cK^\alpha $ endowed with the distance 
\begin{align}\label{distance}
\|\varphi - \tilde \varphi\|_{\alpha} = \sup_{\xi \in \RR^d} \frac{|\varphi(\xi) -
\tilde \varphi(\xi) |}{|\xi|^\alpha}
\end{align}
is a complete metric space (see Proposition 3.10 of \cite{Cannone-Karch-1}).
Moreover, $\cK^\alpha =\{1\}$ for all $\alpha >2$ and the 
following embeddings
(Lemma 3.12 of \cite{Cannone-Karch-1}) hold
\[
\{1\} \subset \cK^\alpha \subset \cK^\beta \subset \cK \enskip \enskip
\mbox{for all $2\ge \alpha \ge \beta > 0$}\,\,.
\]
With this classification on characteristic functions, the global
existence of solution in $\cK^\alpha$ was studied in \cite{Cannone-Karch-1}(see also \cite{morimoto-12}) with the assumption \eqref{index-sing}. 
The space $\cK^\a$ arises in connection with the Fourier image of $P_\a(\RR^d)$ and it is proved 
$\cF(P_\a(\RR^d)) \subset \cK^\a$. However, the inclusion is proper.  Indeed,  it is shown (see Remark 3.16 of \cite{Cannone-Karch-1}) that the function 
$\varphi_\a (\xi) = e^{-|\xi|^\a}$, with $\a \in (0,2)$, belongs to $\cK^\a$ but $p_\a(v)= \cF^{-1} (\varphi_\a)(v)$, the density of $\a$-stable symmetric L\'evy process, 
is not contained in $P_\a(\RR^d)$. 
Hence, the solution obtained in the function space $\cK^\alpha$ does not
represent the moment properties in physics even when it is assumed initially.

In order to capture the precise moment condition in the Fourier
space,  a precise classification on characteristic functions
was introduced in \cite{MWY-2015} ( see also \cite{MWY-2014}). Let % follows:
\begin{align}\label{M-al}
	\wt\cM^\alpha =\{ \varphi \in \cK\,;\, \|\re\varphi - 1\|_{\cM^\alpha}+\|\varphi - 1\|_\a < \infty\}\,,\enskip \a \in (0,2)\,,
\end{align}
where
\begin{align}\label{integral-norm}
	\|\varphi -1\|_{\cM^\a}  = \int_{\RR^d} \frac{|\varphi(\xi)-1 |}{|\xi|^{d+\a} }d\xi\,,
\end{align}
and $\re\vp$ stands for the real part of $\vp(\xi)$. %Accordingly,  the imaginary part of $\vp(\xi)$ is denoted by %$\im\vp$.
Clearly, $\wt\cM^\a\subset\cK^\a$. Moreover, it was shown  in \cite{MWY-2015} that for any $\a\in(0,2)$, $\wt\cM^\a=\cF(P_\a)$.

For $\varphi, \tilde \varphi \in \wt\cM^\a$, put
\begin{align*}
	\|\varphi -\tilde \varphi \|_{\wt\cM^\a}  = \int_{\RR^d} \frac{|\re\varphi(\xi)-\re\tilde \varphi(\xi) |}{|\xi|^{d+\a} }d\xi,
\end{align*}
and we introduce the distance in $\wt\cM^\a$
as
\begin{align}\label{M-dist}
	{dis}_{\a,\b,\e} (\varphi ,\tilde \varphi )= \|\varphi - \tilde \varphi\|_{\wt\cM^\a}+ \|\varphi - \tilde \varphi\|_\b+\|\varphi - \tilde \varphi\|_\b^\e,
\end{align}
where $0<\b<\a<2, {0< \e <1}$.
Then $(\wt\cM^\a,dis_{\a,\b,\e})$ is a complete metric space. In \cite{MWY-2015}, the well-posedness of the 
Cauchy problem  \eqref{bol}-\eqref{initial} is established in $\wt\cM^\a$. 

In this paper, we are devoted to characterizing the Fourier images of  spaces $P_{2k-2+\a}(\RR^d)$.
In \cite{Cho}, Cho characterized the Fourier images of probability measure having finite absolute moment,
without vanishing momentum condition, in terms of the forward difference operator and its iterates. As a modification
of his results, we introduce 
a new classification of characteristic functions defined in terms of 
the symmetric central difference operator and its iterates as follows:

 For any $k\in\bN^+,\a\in[0,2),k+\a>1$, set 
\begin{align}\label{M-k-al}
\cM^\alpha_{k} =\{ \varphi \in \cK\,;\, \int \frac{\Delta^k  \varphi(\xi)}{|\xi|^{d+2k-2+\alpha}}d \xi 
< \infty\}\,,
\end{align}
where 
\begin{align*}
\Delta^1 \varphi(\xi)
=\frac{2-\varphi(\xi)-\varphi(-\xi)}{4}
=\frac{1-\re\varphi(\xi)}{2}
=  \int \sin^2 \frac{v\cdot \xi}{2} dF(v),
\end{align*}
\begin{align*}
\Delta^2 \varphi(\xi) 
&=\frac{6 -4\varphi(\xi)-4\vp(-\xi) +\varphi(2\xi)+\vp(-2\xi)
}{16}\\
&=\frac{3 -4\re\varphi(\xi) +\re\varphi(2\xi) 
}{8}
= \int  \sin^4 \frac{v\cdot \xi}{2} dF(v)
\end{align*}
and generally for $k \in \NN^+$,
\begin{align*}
\Delta^k \varphi(\xi)&=\frac12\sum_{j=0}^kc_{k,j}(\vp(j\xi)+\vp(-j\xi))\\
&=\sum_{j=0}^kc_{k,j}\re\vp(j\xi)=\int \sin^{2k} \frac{v\cdot \xi}{2} dF(v).
\end{align*}
Here, $c_{k,j}$ are the coefficients of the expansion
\begin{align}\label{expansion}
\sin^{2k}\frac x2=\sum_{j=0}^kc_{k,j}\cos(jx) \,\mbox{ for all } x\in\bR,
\end{align}
and an inductive calculation gives 
\[
c_{k,0} =2^{-2k}\binom{2k}{k}\,, \enskip c_{k,j} = {(-1)^j}{2^{-2k+1}}\binom{2k}{k+j}\,, \, j =1,\cdots, k.
\]
For $\varphi, \tilde \varphi \in \cM^\a_k$, put 
\begin{align*}
\|\varphi -\tilde \varphi \|_{\cM^\a_k}  = \int_{\RR^d} \frac{|\Delta^k \varphi(\xi)-
\Delta^k \tilde \varphi(\xi) |}{|\xi|^{d+2k-2+ \a} }d\xi\,,
\end{align*}
and introduce the distance 
\begin{align}\label{M-k-dist}
dis_{k,\a,\b} (\varphi ,\tilde \varphi )= \|\varphi - \tilde \varphi\|_{\cM^\a_k}   + \|\varphi - \tilde \varphi\|_\b\,,
\end{align}
where $0 < \b <\a$ if $k =1$ and $0 < \b <2$ if $k \ge 2$. 

\begin{theo}\label{complete} Let $k\in\bN^+,\a\in[0,2),k+\a>1$, then    
the space $\cM_k^\a$ is a complete metric space endowed with the distance \eqref{M-k-dist}.
Moreover, we have
\begin{align}\label{equivalence-space}
\cM^\a_k  =  \cF(P_{2k-2+\a}(\RR^d)) \,,
\end{align}
and the condition  that $\displaystyle \lim_{n \rightarrow \infty} dis_{k,\a,\b}(\varphi_n, \varphi) = 0$
with $\varphi_n, 
\varphi 
 \in \cM_k^\a$
implies
\begin{align}\label{measure-convergence}
&\lim_{n \rightarrow \infty} \int \psi(v) dF_n(v) = \int \psi(v) dF(v)  
 \end{align}
for any  $\psi \in C(\RR^d)$
 satisfying 
the growth condition $|\psi(v)| \lesssim \la v \ra^{2k-2+\a}$, 
where $F_n = \cF^{-1}(\varphi_n), F = \cF^{-1}(\varphi) \in P_{2k-2+\a}(\RR^d)$.
\end{theo}

\begin{rema}\label{remark-additional}
\begin{itemize}
	\item[(1)] In the case $k=1, 0< \a <2$, the newly defined space $\cM_k^\a$ coincides with
$\wt\cM^\a$. Since the solution in $\wt\cM^\a$ has been well studied in \cite{MWY-2014,MWY-2015}, we will mainly focus on the case $k\ge2,0 \le \a <2$ in this paper.
	\item[(2)] For a probability  measure $F$ on the real line, it has been shown in \cite{kawata} that, if $\int|x|^{2k-2+\a}dF(x)<\infty$, then there exists $C_{k,\a}>0$, which depends  only on $k,\a$, such that
	\begin{align*}
	\int_0^\infty\frac1{t^{1+\a}}\Big\{1-\re\vp(t)+\sum_{j=1}^{k-1}\frac{t^{2j}\vp^{(2j)}(0)}{(2k)!}\Big\}dt=C_{k,\a}\int_{-\infty}^{\infty}|x|^{2k-2+\a}dF(x),
	\end{align*}
where $\varphi =\cF(F)$.	However, this characterization is different from the one given in \eqref{M-k-al}.
\end{itemize}
\end{rema}

Thanks to this  new characterization of $P_{2k-2+\a}$ by its exact Fourier image $\cM_k^\a$, we can obtain the following theorems, that improves the continuity results of solutions established  in \cite{MWY-2015}.

\begin{theo}\label{existence-base-space}
Assume that $b$ satisfies
\eqref{index-sing}. Let $k\in\bN$, $k\ge2$, $\a\in[0,2)$.
If the initial datum $F_0 \in P_{2k-2+\a}(\RR^3)$, %satisfying
%\begin{align}\label{zero-momentum}
%\int v_jdF_0=0,\mbox{ for }j=1,2,3,
%\end{align}
then there exists a unique measure valued solution $F_t \in C([0,\infty), {P}_{2k-2+\a}(\RR^3))$ to the Cauchy problem 
\eqref{bol}-\eqref{initial} which preserves the energy and momentum for all time $t>0$, that is, 
\begin{align}\label{conservation}
\int|v|^2dF_t=\int|v|^2dF_0,\quad\int v_jdF_t=0 \enskip \mbox{ for } j=1,2,3.
\end{align}%, where the continuity with respect to $t$ is according to  the topology deduced in the sense of \eqref{measure-convergence}.
Furthermore, 
if $b$ satisfies \eqref{singular-cross} and 
if $F_0$ is not a single Dirac mass, then $F_t$ admits the density distribution function 
 $f(t,v)$, $dF_t(v) = f(t,v)dv$, satisfying 
$$f \in C((0,\infty);L_{2k-2+\a}^1(\RR^3)\cap H^\infty(\RR^3)).$$
\end{theo}

The proof of the above theorem is given in the Fourier space. In fact, 
by letting $\varphi(t,\xi) = \int_{\RR^3}e^{-iv\cdot \xi} dF_t(v)$ and $\varphi_0= \cF(F_0)$, 
it follows {from} the Bobylev formula \cite{bobylev,bobylev2} that
the Cauchy problem \eqref{bol}-\eqref{initial} is reduced to 
\begin{equation}\label{c-p-fourier}
\left \{ 
\begin{array}{l}\dis \partial_t \varphi(t,\xi)
=\int_{\SS^2}b\Big(\frac{\xi \cdot \sigma}{|\xi|}\Big) \Big( \varphi(t,\xi^+)\varphi(t, \xi^-) - \varphi(t, \xi)
\varphi(t,0)\Big) d\sigma, \\\\
\dis \varphi(0,\xi)=\varphi_0(\xi), \enskip \mbox{where} \enskip
\dis \xi^\pm = \frac{\xi}{2} \pm \frac{|\xi|}{2} \sigma\,.
\end{array}
\right.
\end{equation}
By  Theorem \ref{complete}, to prove Theorem \ref{existence-base-space}
it suffices to show 
\begin{theo}\label{fourier-space}
Assume that $b$ satisfies { \eqref{index-sing}}. Let $k\in\bN$, $k\ge2$, $\a\in[0,2)$, $\b\in(\a_0,2)$. 
If the initial datum $\varphi_0=\cF(F_0)\in\cM_k^\a$, with $F_0$ satisfying \eqref{mean}, 
 then there exists a unique classical solution $\varphi(t,\xi) \in C([0,\infty), \cM_k^\a)$
to the Cauchy problem \eqref{c-p-fourier}. Moreover, for any $T>0$ and $0\le s<t\le T$, we have
\begin{align}\label{fourier-continuity}
||\varphi(t,\cdot)- \varphi(s,\cdot)||_\b&\le |t-s|\cdot  e^{\lam_\b T}||1-\vp_0||_{\b},\\
||\varphi(t,\cdot)- \varphi(s,\cdot)||_{\cM_k^\a}&\lesssim |t-s| \cdot\sup_{\tau\in[0,T]}\int|v|^{2k-2+\a}dF_\tau\,,\label{continuity-M}
\end{align}
where $\lam_\b>0$ is a constant defined as
\begin{align*}
\lam_\b=\int_0^{\pi/2}b(\cos\theta)\Big(\sin^\b\frac\theta2+\cos^\b\frac\theta2-1\Big)\sin\theta d\theta. 
\end{align*}
\end{theo}
\begin{rema}
	\eqref{fourier-continuity} and \eqref{continuity-M} imply that
	\begin{align}
	dis_{k,\a,\b}(\vp(t,\xi),\vp(s,\xi))\le C_{T,\vp_0}|t-s|,
	\end{align}
	where $C_{T,\vp_0}>0$ only depends on $T$ and the initial data.
\end{rema}

\section{Characterization of $P_{2k+\a}(\bR^3)$}\label{s2}
\setcounter{equation}{0}

\begin{prop}\label{chara-2}
Let $k\in\bN,\a \in [0,2),k+\a>1$ and let $\cM^{\a}_{k}$ be a subspace of $\cK = \cF(P_0(\RR^d))$ defined by \eqref{M-al}.
Then we have the formula  \eqref{equivalence-space}. 
Furthermore, for $M \in[1, \infty]$, if we put 
\begin{align}\label{uniform-constant}
c_{\alpha,d,M,k} = \int_{\{|\zeta| \le M\}} \frac{\sin^{2k} (\ve_1 \cdot \zeta/2)}{|\zeta|^{d+2k-2+\alpha}} d \zeta >0,
\end{align}
and if $F = \cF^{-1}(\varphi)$ for $\varphi \in \cM_k^\a$, then for any $R >0$ we have
\begin{align}\label{uniform-moment-est}
\int_{\{|v| \ge R \}} |v|^{2k -2+ \alpha} dF(v) \le \frac{1}{c_{\alpha,d,1,k}}
\int_{\{|\xi| \le 1/R\}} \frac{\Delta^k \varphi(\xi)}{|\xi|^{d+2k-2+ \alpha}} d \xi.
\end{align}
Moreover,
\begin{align}\label{moment-es}
\int_{\RR^d} |v|^{2k-2+\alpha} dF(v) \le \frac{1}{c_{\alpha,d, \infty,k}}
\int_{\RR^d} \frac{\Delta^k \varphi(\xi)}{|\xi|^{d+2k-2+\alpha}} d \xi  \,.
\end{align}
\end{prop}

\begin{proof}
Note %Since $|\varphi(\xi)| \le \varphi(0) =1$, we have
\begin{align*}
\int_{\{|\xi| \le M/R\}} \frac{\Delta^k \varphi(\xi) }{|\xi|^{d+2k-2+ \alpha}} d \xi % &\ge 
%\int_{\{|\xi| \le M/R\}}\frac{\mbox{Re}~ (1-\varphi(\xi))}{|\xi|^{d+\alpha}} d \xi \\
&= \int_{\RR_v^d} \Big(\int_{\{|\xi| \le M/R\}}\frac{\sin^{2k} (v\cdot \xi/2)}{|\xi|^{d+2k-2+\alpha}} d \xi \Big)
dF(v).
\end{align*}
By the change of variable $|v| \xi = \zeta$ and by using
the invariance of the rotation, we have
\begin{align*}
\int_{\{|\xi| \le M/R\}} \frac{\sin^{2k} (v\cdot \xi/2)}{|\xi|^{d+2k-2+\alpha}} d \xi
&=|v|^{2k-2+\alpha} \int_{\{|\zeta| \le M |v|/R\}} \frac{\sin^{2k} (\ve_1 \cdot \zeta/2)}{|\zeta|^{d+2k-2+\alpha}} d \zeta \\
& \ge |v|^{2k-2+\alpha} {\bf 1}_{\{|v| \ge R\}} c_{\alpha,d,M,k} \,,
\end{align*} 
which yields \eqref{uniform-moment-est}, with the choice of $M=1$. 
By letting  $M \rightarrow \infty$  and   $R \rightarrow 0$, we obtain  \eqref{moment-es}.  %In order to complete the proof of \eqref{inclusion-spa}, it remains to show \eqref{mean} when  $\a \ge 1$. 
%Suppose that 
%\[
%\exists \varphi \in \cM^\a \enskip s.t. \enskip a := \int v d F(v) \ne 0\,, \enskip F = \cF^{-1}(\varphi).
%\]
%Since $F(\cdot + a)$ belongs to $P_\a(\RR^d)$ and satisfies \eqref{mean}, it follows {from}  Proposition \ref{chara-1} that its Fourier transform  
%$$\varphi_a(\xi)  = \int e^{-iv\cdot \xi} dF(v+a)= e^{ia \cdot \xi} \varphi(\xi)$$ 
%belongs to $\cM^\a$,  that is,
%\begin{align*}
%\infty > \int_{\RR^d} \frac{|1-e^{ia \cdot \xi} \varphi(\xi)|}{|\xi|^{d+\alpha}} d \xi = \int_{\RR^d} \frac{|(e^{-ia \cdot \xi}-1) + (1-\varphi(\xi))|}{|\xi|^{d+\alpha}} d \xi\,.
%\end{align*}
%Since $\varphi \in \cM^\a$, we obtain $\int_{\RR^d} \frac{|e^{-ia \cdot \xi}-1|}{|\xi|^{d+\alpha}} d \xi< \infty$, which contradicts $a \ne 0$. In
%fact, by the rotation, we can assume $a = (0,\cdots,0, |a|)$ and hence 
%\[
%\infty > \int_{\RR^d} \frac{|e^{-ia \cdot \xi}-1|}{|\xi|^{d+\alpha}} d \xi
%\ge |\SS^{d-2}|\int_0^{\pi/2}\Big( \int_0^{1/|a|} \frac{\sin (\rho |a| \cos \theta)}{\rho^{1+\a}} d\rho\Big) d \theta\,.
%\]   
The formula  \eqref{equivalence-space} is now obvious since
\begin{equation}\label{opposite}
\lim_{M \rightarrow \infty} \int_{\{|\xi| \le M \}} \frac{\Delta^k \varphi(\xi) }{|\xi|^{d+2k-2+ \alpha}} d \xi
\leq c_{\alpha,d,\infty ,k}\int |v|^{2k-2+\a} dF(v). 
\end{equation}
 
\end{proof}

We are now ready to prove Theorem \ref{complete}.

\begin{proof}[Proof of Theorem \ref{complete}]
Suppose that $\{\varphi_n \}_{n=1}^\infty \subset \cM_k^\a$ satisfies
\[
dis_{k,\a,\b} (\varphi_n, \varphi_m) \rightarrow 0 \enskip (n,m \rightarrow \infty) \enskip .
\]
Since it follows {from} Proposition 3.10 of \cite{Cannone-Karch-1} that $\cK^\b$ is a complete metric space,
we have the limit (pointwise convergence)
$$\varphi(\xi) = \lim_{n \rightarrow \infty } \varphi_n(\xi)  \in \cK^\b \subset \cK.
$$
For any fixed $R >1$ we have 
\[
\int_{\{R^{-1}\le |\xi| \le R\}}\frac{|\D^k\varphi_n (\xi)|}{|\xi|^{d+2k-2+\a}} d\xi \le \sup_{n} \|\varphi_n \|_{\cM_k^\a} < \infty.
\]
Taking the limit with respect to $n$ and letting $R \rightarrow \infty$, we have $\varphi \in \cM_k^\a$.
Now it is easy to see that  $dis_{k,\a,\b} (\varphi_n, \varphi) \rightarrow 0$.

Suppose that, for $F_n, F \in P_\a(\RR^d)$,  we have 
$$
\varphi_n = \cF(F_n), \varphi = \cF(F) \in \cM_k^\a,\enskip \mbox{and} \enskip 
\displaystyle \lim_{n \rightarrow \infty} dis_{k,\a,\b}(\varphi_n, \varphi) = 0\,.
$$ 
Note that for $R >1$
\[
\int_{\{|\xi| \le 1/R\}} \frac{\D^k\varphi_n (\xi)}{|\xi|^{d+\alpha}} d \xi  
\le \int_{\{|\xi| \le 1/R\}} \frac{\D^k\varphi(\xi)}{|\xi|^{d+\alpha}} d \xi  + \|\varphi_n -\varphi\|_{\cM_k^\a} \,. 
\]
It follows from \eqref{uniform-moment-est} that for any $\varepsilon >0$ there exist $R >1$ and $N \in \NN$ such that 
 \[  
\int_{\{|v| \ge R \}} |v|^{2k-2+\a} dF_n(v) +  \int_{\{|v| \ge R \}} |v|^{2k-2+\a} dF(v) < \varepsilon\, \enskip \mbox{if $n \ge N$}.
\]
This shows \eqref{measure-convergence} because $\varphi_n \rightarrow \varphi$ in $\cS'(\RR^d)$ and so
$F_n \rightarrow F$ in $\cS'(\RR^d)$. This completes the proof of the
theorem.
\end{proof}

\section{Proof of Theorem \ref{fourier-space} }\label{s-34}

The main purpose of this section concerns with the continuity of  solutions in the new classification of  characteristic functions.
We only need to prove  Theorem  \ref{fourier-space} because Theorem \ref{existence-base-space} follows by using Theorem \ref{complete}.

\begin{proof}
Since $k\ge2$, we have $\vp_0\in\cM_k^\a\subset\cK^2\subset \wt\cM^{\b_1}$, for some $\b_1\in(\b,2)$.
By Theorem 1.8 of \cite{MWY-2015}, we can obtain a unique classical solution $\vp(t,\xi)\in C([0,\infty),\wt\cM^{\b_1})$ 
corresponding to the initial data $\vp_0$.
Since $F_0\in P_{2k-2+\a}(\bR^3)$, by the Corollary 1.7 of 
\cite{MWY-2015}  and    \eqref{opposite} , it is easy to see that, for any $t\ge0$, $\vp(t,\xi)\in\cM_k^\a$. More precisely, there exists a constant $C_T>0$ such that
\begin{align}
\sup_{s \le \tau \le t} \|\vp(\tau) \|_{\cM^\alpha_k} \le C_T\enskip\mbox{ for all } s,t\in[0,T].
\end{align}

To complete the proof of Theorem  \ref{fourier-space},
it remains to prove \eqref{fourier-continuity} and \eqref{continuity-M}. 
The first estimate \eqref{fourier-continuity} is a direct consequence of the formula 
\begin{align*}
\vp(t,\xi) -\vp(s,\xi) 
= \int_s^t \int_{\SS^2}b\Big(\frac{\xi \cdot \sigma}{|\xi|}\Big) \Big( \vp(\tau, \xi^+)\vp (\tau, \xi^-) - \vp (\tau, \xi)
\Big) d\sigma d \tau,
\end{align*}
and Lemma 2.2 of \cite{morimoto-12}. The second one \eqref{continuity-M} follows {from} 
the following calculation which is a variant of Lemma 2.2 of \cite{morimoto-12} and Lemma 3.4 of \cite{MWY-2014}
(see also the proof of (1.23) in \cite{MWY-2015}).  

For any $k\ge2,\a\in[0,2)$, we have
\begin{align}
	&\int\frac{|\D^k\vp(\xi,t)-\D^k\vp(\xi,s)|}{|\xi|^{3+2k-2+\a}}d\xi \nonumber\\
	&=\int\frac{d\xi}{|\xi|^{3+2k-2+\a}}\Big|\sum_{j=1}^kc_{k,j}(\re\vp(j\xi,t)-\re\vp(j\xi,s))\Big|\nonumber\\
	&=\int\frac{d\xi}{|\xi|^{3+2k-2+\a}}\Big|\sum_{j=1}^kc_{k,j}\re\int_s^t\int_{\bS^2}b\Big(\frac{\xi\cdot\s}{|\xi|}\Big)\nonumber\\
	&\,\qquad\qquad\qquad\cdot\big(\vp(j\xi^+,\tau)\vp(j\xi^-,\tau)-\vp(j\xi,\tau)\big)d\s d\tau\Big|.\label{split-equation}
	\end{align}
	
    As in \cite{morimoto-12}, we put  $\zeta = \big(\xi^+ \cdot \frac{\xi}{|\xi|}\big) \frac{\xi}{|\xi|}$ and 
	consider $\tilde \xi^+ = \zeta -(\xi^+-\zeta)$, 
	which is
	symmetric to $\xi^+$ on $\SS^2$, see Figure 1.	
	\begin{figure}[bht]\label{f-2-3}
		\begin{center}
			\includegraphics[width=9cm]{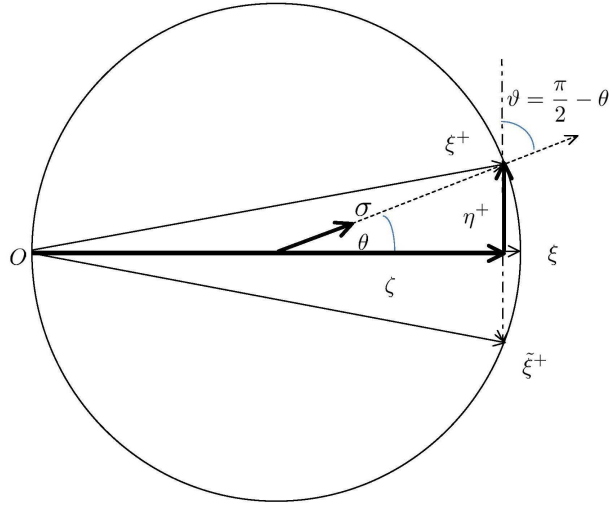}
			\caption{$\theta = \la \frac{\xi}{|\xi|} , \sigma \ra $, \enskip  $\vartheta = \la\frac{\eta^+}{|\eta^+|}, \sigma\ra$, $\eta^+ =\xi^+ -\zeta$. }
		\end{center}
	\end{figure}	
	Then, we split \eqref{split-equation} into three parts:
	\begin{align}
	I_{j,1}&=\frac12\int b\Big(\frac{\xi\cdot\s}{|\xi|}\Big)(\vp(j\xi^+)+\vp(j\tilde\xi^+)-2\vp(j\zeta))d\s,\\
	I_{j,2}&=\int b\Big(\frac{\xi\cdot\s}{|\xi|}\Big)(\vp(j\zeta)-\vp(j\xi))d\s,\\
	I_{j,3}&=\int b\Big(\frac{\xi\cdot\s}{|\xi|}\Big)\vp(j\xi^+)(\vp(j\xi^-)-1)d\s.
	\end{align}
Summing over $j$, we have 
	\begin{align*}
	\sum_{j=1}^kc_{k,j}\re I_{j,1}&=\frac12\int b\Big(\frac{\xi\cdot\s}{|\xi|}\Big)\int\Big(\sin^{2k}\frac{\xi^+\cdot v}{2}+\sin^{2k}\frac{\tilde\xi^+\cdot v}{2}\\
	&\quad -2\sin^{2k}\frac{\zeta\cdot v}{2}\Big)dF_\tau d\s,\\
	\sum_{j=1}^kc_{k,j}\re I_{j,2}&=\int b\Big(\frac{\xi\cdot\s}{|\xi|}\Big)\int\Big(\sin^{2k}\frac{\zeta\cdot v}{2}-\sin^{2k}\frac{\xi\cdot v}{2}\Big)dF_\tau d\s.
	\end{align*}	
	By a proper transformation to the variable $\xi$, that is , $\eta=|v|\xi$, and then use $\xi$ again to replace $\eta$, we have
	
	\begin{align*}
	&\int\frac{1}{|\xi|^{3+2k-2+\a}}\Big|\int_s^t \sum_{j=1}^kc_{k,j}\re I_{j,1}d\tau \Big|d\xi\\
	&\le\frac12\int_s^t\int\int_{\bS^2} b\Big(\frac{\xi\cdot\s}{|\xi|}\Big)\Big|\sin^{2k}\frac{\xi^+\cdot \bfe_1}{2}+\sin^{2k}\frac{\tilde\xi^+\cdot \bfe_1}{2}-2\sin^{2k}\frac{\zeta\cdot \bfe_1}{2}\Big|d\s d\xi\\
	&\qquad\cdot\int |v|^{2k-2+\a}dF_\tau d\tau.
	\end{align*}
	If we put $A = \zeta\cdot \bfe_1/2$ and $B = \eta^+ \cdot \bfe_1/2$,  ($ \eta^+ = \xi^+ - \zeta$), then 
\begin{align*}
	&\sin^{2k}\frac{\xi^+\cdot \bfe_1}{2}+\sin^{2k}\frac{\tilde\xi^+\cdot \bfe_1}{2}-2\sin^{2k}\frac{\zeta\cdot \bfe_1}{2}\\
&= (\sin A \cos B + \sin B \cos A)^{2k} +(\sin A \cos B - \sin B \cos A)^{2k}- 2\sin^{2k} A\\
& =  2\sin^{2k} A \Big ( (\cos^2 B)^k -1 \Big) \\
&\qquad \qquad + 2 \sum_{j=1}^k \left( \begin{array}{c}
2k\\
2j
\end{array}\right) \sin^{2j} B   \sin^{2k-2j}  A  \cos^{2k-2j} B \cos^{2j} A \,.
\end{align*}
Since 
	\begin{align*}
		\sin^2  \frac{\eta^+\cdot\bfe_1}{2} \lesssim|\xi|^2\sin^2\frac\theta2,
	\end{align*}
		we have
	\begin{align*}
		&\Big|\sin^{2k}\frac{\xi^+\cdot \bfe_1}{2}+\sin^{2k}\frac{\tilde\xi^+\cdot \bfe_1}{2}-2\sin^{2k}\frac{\zeta\cdot \bfe_1}{2}\Big|\\
		&\lesssim
		\min \Big\{|\xi|^2\sin^2\frac\theta2,\, 1\Big\}\cdot {\bf 1}_{\{|\xi|\ge1\}}+
		|\xi|^{2k}\sin^2\frac\theta2\cdot {\bf 1}_{\{|\xi|<1\}}.
	\end{align*}	
	For $|\xi|>1$,
	\begin{align*}
	&\int_{|\xi|>1}\frac{1}{|\xi|^{3+2k-2+\a}}\Big|\int_s^t \sum_{j=1}^kc_{k,j}\re I_{j,1}d\tau \Big|d\xi\\
	&\lesssim\int_{|\xi|>1}\frac{1}{|\xi|^{3+2k-2+\a}}\int_0^{\pi/2}b(\theta)\min\{|\xi|^2\sin^2(\theta/2),\, 1\}\sin\theta d\theta d\xi\\
	&\lesssim\int_{|\xi|>1}\frac{|\xi|^{\a_0}}{|\xi|^{3+2k-2+\a}}d\xi<\infty.
	\end{align*} 
	The case $|\xi|<1$ is easier. Hence, we proved
	\begin{align}\label{same-estimate}
	\int&\frac{1}{|\xi|^{3+2k-2+\a}}\Big|\int_s^t \sum_{j=1}^kc_{k,j}\re I_{j,1}d\tau \Big|d\xi\nonumber\\
	&\lesssim |t-s|\cdot\sup_{\tau\in[0,T]}\int|v|^{2k-2+\a}dF_\tau.
	\end{align}
	
	The second term $I_{j,2}$ is similar.
	Now let's consider the last term $I_{j,3}$. Recall that the solution conserves the momentum, i.e.
	\begin{align}\label{zeromomentum}
	\int vdF_t=0 \mbox{ for all } t\ge0.
	\end{align}
We have
	\begin{align*}
	\re I_{j,3}
	&=\int b\left(\frac{\xi\cdot\s}{|\xi|}\right)\re \Big(\vp(j\xi^+)(\vp(j\xi^-)-1)\Big)d\s\\
	&=\int b\left(\frac{\xi\cdot\s}{|\xi|}\right)\Big(\iint \re\Big( e^{-i(j\xi^+\cdot v)} (e^{-i(j\xi^-)\cdot w}-1)\Big) 
dF_\tau(v) dF_\tau(w)\Big)d\s  \\
	&=\int b\left(\frac{\xi\cdot\s}{|\xi|}\right)
\Big(\iint 
\big\{\cos(j(\xi^+\cdot v+\xi^-\cdot w))\\
&\qquad\qquad\qquad \qquad \qquad \qquad -\cos(j\xi^+\cdot v)\big\}
dF_\tau(v) dF_\tau(w)\Big)d\s. \\
\end{align*}
It follows from \eqref{expansion} that
\begin{align*}
	\sum_{j=1}^kc_{k,j}\re I_{j,3}&=\int b\Big(\frac{\xi\cdot\s}{|\xi|}\Big)\Big( \iint \Big(\sin^{2k}\frac {\xi^+\cdot v+\xi^-\cdot w}2\\
%	&\quad-(\xi^-\cdot w)k\sin^{2k-1}\frac{\xi^+\cdot %v}2\cos\frac{\xi^+\cdot v}2\\
	&\quad-\sin^{2k}\frac{\xi^+\cdot v}2\Big)dF_\tau(v)dF_\tau(w)\Big)d\s.
	\end{align*}
	Denoting $x=\frac{\xi^+\cdot v}{2},y=\frac{\xi^-\cdot w}{2}$, we see
\begin{align*}
&\sin^{2k}\frac {\xi^+\cdot v+\xi^-\cdot w}2-\sin^{2k}\frac{\xi^+\cdot v}2\\
&\quad =\Big(\sin x \cos y + \sin y \cos x\Big)^{2k} - \sin^{2k} x\\
&\quad=2k \sin y \sin^{2k-1} x \cos x\\
&\qquad + 2k \sin y(\cos^{2k-1}y-1)\sin^{2k-1} x\cos x \\
&\qquad + \sum_{\ell =2}^{2k}\left(\begin{array}{c}2k\\
\ell\end{array}\right)  \sin^{\ell}y 
\cos^{2k-\ell}y \sin ^{2k-\ell} x \cos ^ \ell x\\
&\qquad +\sin^{2k} x(\cos^{2k}y-1)\\
&\quad = J_{3,1} + J_{3,2} + J_{3,3}+J_{3,4}.
\end{align*}
Since the solution conserves the momentum,
we obtain 
\begin{align*}
&\left|\iint J_{3,1} dF_\tau(v)dF_\tau(w)\right|\\
&= 2k \left|\int \sin^{2k-1}\frac{\xi^+\cdot v}{2}\cos \frac{\xi^+\cdot v}{2} dF_\tau(v) \int \Big( \sin \frac{\xi^-\cdot w}{2} -\frac{\xi^-\cdot w}{2} \Big) dF_\tau(w)\right|.
\end{align*}
Since $|z- \sin z| \lesssim |z| \min\{|z|, 1\}$, by the change of the variable $|v| \xi \rightarrow \zeta$, we have
\begin{align*}
&\int\frac{d\xi}{|\xi|^{3+2k-2+\a}}\int b\left(\frac{\xi\cdot\s}{|\xi|}\right)\left|\iint J_{3,1} dF_\tau(v)dF_\tau(w)\right|d\s\\
&\lesssim \int\frac{\min\{|\zeta|^{2k-1}, 1\}d\zeta}{|\zeta |^{3+2k-2+\a}}\int b\left(\cos \theta \right)\theta^{\max\{\alpha_0, 1\}} \sin \theta
d \theta\\
&\qquad \times 
\int |v|^{2k-2+\alpha} \Big( \int \Big(\frac{ |\zeta||w|}{|v|}\Big)^{\max\{\alpha_0, 1\}} dF_\tau(w)\Big) dF_\tau(v)\\
&\lesssim\int \la v\ra^{2k-2+\alpha}   dF_\tau(v) \int \la w \ra^2 dF_\tau(w).
%\sin^{2k-1}\frac{\zeta ^+\cdot \frac{v}{|v|} }{2}
\end{align*}
We have the same upper bound for the integrals corresponding to $J_{3,2}$ and $J_{3,4}$. 
Furthermore, if one use another change of variable $|w| \xi \rightarrow \zeta$ for terms with $\ell \ge k$, then 
\begin{align*}
&\int\frac{d\xi}{|\xi|^{3+2k-2+\a}}\int b\left(\frac{\xi\cdot\s}{|\xi|}\right)\left|\iint J_{3,3} dF_\tau(v)dF_\tau(w)\right|d\s\\
&\lesssim \int \frac{d\zeta}{|\zeta |^{3+2k-2+\a}}\int b\left(\cos \theta \right)\theta^{\alpha_0} \sin \theta
d \theta   \\
&\qquad \times\iint |v|^{2k-2+\a}\sum_{\ell=2}^{k-1}\Big( |\zeta|^{2k-\ell} {\bf 1}_{\{|\zeta| \le 	1\}} + {\bf 1}_{\{|\zeta| >1\}}\Big)\\
&\qquad \times \Big(\Big(\frac{|w||\zeta|}{|v|}\Big)^{\ell} {\bf 1}_{\{|\zeta| \le 	1\}} + \Big(\frac{|w||\zeta|}{|v|}\Big)^{\a_0} {\bf 1}_{\{|\zeta| >1\}}
\Big) dF_\tau(v)dF_\tau(w)\\
&\quad + \int \frac{d\zeta}{|\zeta |^{3+2k-2+\a}}\int b\left(\cos \theta \right)\theta^{\alpha_0} \sin \theta
d \theta \\
&\qquad \times \iint |w|^{2k-2+\a}  \sum_{\ell=k}^{2k}\Big( \Big(\frac{|v||\zeta|}{|w|}\Big)^{2k-\ell} {\bf 1}_{\{|\zeta| \le 	1\}} + {\bf 1}_{\{|\zeta| >1\}}\Big)\\
&\qquad \times  \Big(|\zeta|^{\ell} {\bf 1}_{\{|\zeta| \le 	1\}} + |\zeta|^{\a_0} {\bf 1}_{\{|\zeta| >1\}}
\Big)   dF_\tau(v)dF_\tau(w)  \\
&\lesssim \iint \Big(\la v \ra^{2k-2+\a}\la w \ra^2 + \la w \ra^{2k-2+\a}\la v\ra^2\Big)dF_\tau(v)dF_\tau(w) .
\end{align*}
Thus, we have similar  estimate as \eqref{same-estimate}, namely,
\begin{align*}
	\int\frac{1}{|\xi|^{3+2k-2+\a}}\Big|\int_s^t \sum_{j=1}^kc_{k,j}\re I_{j,3}d\tau \Big|d\xi
	\lesssim|t-s|\cdot\sup_{\tau\in[0,T]}\int\la w\ra^{2k-2+\a}dF_{\tau}(w).
	\end{align*}
\end{proof}

\section{Proof of Theorem  \ref{existence-base-space}}
The existence and uniqueness of the solution to the Cauchy problem \eqref{bol}-\eqref{initial} follow from Theorem \ref{fourier-space} and the smoothing effect is proved by the Corollary 1.10 of \cite{MWY-2015}. To finish the proof, it remains to show 
\begin{align*}
\cF^{-1} (\varphi(t)) = F_t = f(t,v) \in C((0, \infty); L^1_{2k-2+ \a}(\RR^3))
\end{align*}
if $b$ has a singularity \eqref{singular-cross}.  

Indeed, it follows from the smoothing effect that
for any $0 < t_0 < T < \infty$ and for any $N >0$, there exists a constant $C_{N,t_0,T}>0$ such that
\[
\sup_{t_0 \le \tau \le T} \int \la \xi \ra^N |\varphi(\tau, \xi)|^2 d \xi \le  C_{N, t_0, T}\,.
\]
Let $t_1 \in (t_0, T)$ and $\varepsilon >0$. Since $\varphi(t_1, \xi) \in \cM_k^\a$, there exists 
an $R= R_\varepsilon >1$ such that 
\[
\int_{|\xi| < 1/R}  \frac{\Delta^k \vp(t_1,\xi) }{|\xi|^{3+2k-2+\a}} d\xi <  c_{\a,3,1,k}\frac{\varepsilon}{2}
\]
and moreover it follows from \eqref{continuity-M} that there exists a $\delta >0$ such that
if $|t-t_1| < \delta$ then 
\[
\int_{|\xi| < 1/R}  \frac{\Delta^k \vp(t,\xi) }{|\xi|^{3+2k-2+\a}} d\xi <   c_{\a,3,1,k}\varepsilon.
\]
By means of \eqref{uniform-moment-est},  we have 
$\int_{\{|v| \ge R \}} |v|^{2k + \alpha-2} f(t,v) dv   < \varepsilon$ if $|t-t_1| < \delta$. 
Therefore, if $|t-t_1| <\delta$ then
\begin{align*}
&\int \la v \ra^{2k-2 + \a} |f(t,v) -f(t_1,v)|dv \\
&\quad < \la R \ra^{2k-2+\a} \int_{|v| <R} |f(t,v) -f(t_1,v)|dv
+ 4 \varepsilon . 
\end{align*}
On the other hand, for any $M >1$ we have
\begin{align*}
&\sup_{v} |f(t,v) -f(t_1,v)| \le \int |\vp(t, \xi) -\vp(t_1, \xi) | d\xi \\
& \le \Big(\int_{|\xi| \ge M } \la \xi \ra^{-4} d\xi \Big) ^{1/2} 
\Big( \int \la \xi \ra^4 |\vp (t, \xi) - \vp(t_1, \xi)|^2d \xi \Big)^{1/2}\\
&\quad \qquad + \frac{4 \pi M^3 }{3 }\sup_{|\xi| \le M}  |\vp (t, \xi) - \vp(t_1, \xi)|\,.
\end{align*}
We can conclude the continuity since $\varphi(t,\xi) \in C([0,\infty); \cK^2)$.

\bigskip

\noindent

\bigskip
\noindent
{\bf Acknowledgements.} 
The research of the first author was supported in part by National Research Foundation of Korea
Grant funded by the Korean Government (No. 20150301).
The research of the second  author was supported in part
by  Grant-in-Aid for Scientific Research No.25400160,
Japan Society for the Promotion of Science. The research of the fourth author was
supported in part by the General Research Fund of Hong Kong, CityU No. 11303614.

%%%%%%%%%%%%%%%%%%%%%%%%%%%%%%%%%%%%%%%%%%%%%%%%%%%%%%%%%%%%%
%%%%%%%%%%%%%%%%%%%%%%%%%%%%%%%%%%%%%%%%%%%%%%%%%%%%%%%%%%%%%


\begin{thebibliography}{99}




\bibitem{ADVW} R. Alexandre, L. Desvillettes, C. Villani and B. Wennberg,
Entropy  dissipation and long-range interactions,  {\it Arch.
Rational Mech. Anal.} {\bf 152} (2000),  327-355.


\bibitem{bobylev}
A. V. Bobylev, The method of the Fourier transform in the theory of the Boltzmann equation for Maxwell molecules, {\it Dokl. Akad. Nauk SSSR }{\bf 225}(1975), no. 6,1041-1044; translation in
{\it Soviet Physics Dokl. }{\bf 20}(1975), no. 12,820-822.
\bibitem{bobylev2}
A. V. Bobylev, The theory of the nonlinear spatially uniform Boltzmann equation for Maxwell molecules, {\it Soviet Sci. Rev. Sect. C Math. Phys. Rev. }{\bf 7}(1988), 111-233.



%\bibitem{bobylev-Cercignani}
%A. V. Bobylev and C. Cercignani, Self-similar solutions of the Boltzmann equation and their
%applications, {\it J. Statist. Phys.}{\bf 106}(2002), 1039-1071.


\bibitem{Cannone-Karch-1} M. Cannone and G. Karch,
{Infinite energy solutions to the homogeneous Boltzmann equation},
{\it Comm. Pure Appl. Math.} {\bf 63} (2010), 747-778. 

%\bibitem{Cannone-Karch-2} M. Cannone and G. Karch,
%On self-similar solutions to the homogeneous Boltzmann equation, { \it Kinetic and Related Models}, {\bf 6} (2013), 801 - 808.


\bibitem{Carlen-Gabetta-Toscani} E. A. Carlen, E. Gabetta
 and G. Toscani, Propagation of smoothness and the rate of exponential convergence to equilibrium for a spatially homogeneous Maxwellian gas, {\it Comm. Math. Phys.},
{\bf 199} (1999) 521-546.

\bibitem{Cho} {Y. K. Cho}, Absolute moments and Fourier-based
probability metrics, {\it preprint}.

\bibitem{Gabetta-Toscani-Wennberg}
E. Gabetta, G. Toscani and B. Wennberg, Metrics for probability distributions and the rend to equilibrium for solutions of the Boltzmann equation, {\it J. Statist. Phys}, {\bf 81} (1995), 901-934.

\bibitem{kawata}T. Kawata, 
Fourier analysis in probability theory,
Probability and Mathematical Statistics, No. 15. Academic Press, New York-London, 1972.


\bibitem{morimoto-12}Y. Morimoto, 
A remark on Cannone-Karch solutions 
to the homogeneous Boltzmann equation for Maxwellian molecules, 
{ \it Kinetic and Related Models}, {\bf 5} (2012), 551-561.

\bibitem{MUXY-DCDS}
Y. Morimoto, S. Ukai, C.-J. Xu and T. Yang, {Regularity of solutions to the
spatially homogeneous Boltzmann equation without angular cutoff},
{\it Discrete and Continuous Dynamical Systems -
 Series A} {\bf 24} (2009), 187--212.






\bibitem{MWY-2014}
Y. Morimoto, S. Wang and T. Yang, {A new characterization and global regularity
	of infinite energy solutions to the homogeneous Boltzmann equation},
{\it J. Math. Pures Appl.} {\bf 103} (2015), 809--829.


\bibitem{MWY-2015}
Y. Morimoto, S. Wang and T. Yang, {Moment classification of infinite energy solutions to the homogeneous Boltzmann equation}, 
{\it  online, Analysis and Applications (2015)
DOI: 10.1142/S0219530515500232}, http://arxiv.org/abs/1506.06493

\bibitem{MY} Y. Morimoto and T. Yang,
Smoothing effect of the homogeneous Boltzmann equation with measure valued initial datum, 
{\it Ann. Inst. H. Poincar\'e Anal. Non Lin\'eaire},
{\bf 32} (2015), 429-442.







\bibitem{toscani-villani}G. Toscani and C. Villani, Probability metrics and uniqueness of the solution to the Boltzmann equations for Maxwell gas, {\it J. Statist. Phys.,} {\bf 94} (1999), 619-637.









\bibitem{villani}C. Villani, {On a new class of weak solutions to the spatially
homogeneous Boltzmann and Landau equations}, {\it Arch. Rational
 Mech. Anal.,} {\bf 143} (1998), 273--307.


\bibitem{villani2} C. Villani, A review of mathematical
topics in collisional kinetic theory. In: Friedlander S.,
Serre D. (ed.),
Handbook of Fluid Mathematical Fluid Dynamics, Elsevier Science   (2002).



\bibitem{villani3} C. Villani,
Topics in optimal transportation. Graduate Studies in Mathematics, {\bf 58}. American Mathematical Society, Providence, RI, (2003) .  %xvi+370 pp. 

\end{thebibliography}
\end{document}